\documentclass[11pt,reqno]{amsart}

\usepackage[margin=1in]{geometry}
\usepackage{mathtools}
\usepackage{amssymb}
\usepackage{booktabs}
\usepackage{tabularx}
\usepackage{array}
\usepackage{enumitem}
\usepackage{microtype}
\usepackage[hidelinks]{hyperref}
\hypersetup{
  pdftitle={Fixed-lowering sl2-triples for Laurent-shift operators: exact stencil endpoints and recurrence locality},
  pdfauthor={Kyle Singh},
  pdfkeywords={sl2-triple, difference operator, Laurent shift algebra, Delta-Appell polynomial, d-orthogonality, Charlier polynomial, recurrence bandwidth}
}

\numberwithin{equation}{section}

\newtheorem{theorem}{Theorem}[section]
\newtheorem{proposition}[theorem]{Proposition}
\newtheorem{corollary}[theorem]{Corollary}
\newtheorem{lemma}[theorem]{Lemma}
\theoremstyle{definition}
\newtheorem{definition}[theorem]{Definition}
\newtheorem{example}[theorem]{Example}
\theoremstyle{remark}
\newtheorem{remark}[theorem]{Remark}

\newcommand{\K}{\mathbb K}
\newcommand{\R}{\mathbb R}
\newcommand{\Z}{\mathbb Z}

\newcommand{\algA}{\mathcal A}
\newcommand{\calP}{\mathcal P}
\newcommand{\calL}{\mathcal L}
\newcommand{\supp}{\operatorname{supp}}
\newcommand{\ev}{\operatorname{ev}}
\newcommand{\Span}{\operatorname{span}}
\newcommand{\id}{\operatorname{Id}}
\newcommand{\ad}{\operatorname{ad}}
\newcommand{\fall}[2]{#1^{\underline{#2}}}

\title[Fixed-lowering $\mathfrak{sl}_2$-triples]{Fixed-lowering $\mathfrak{sl}_2$-triples for Laurent-shift operators: exact stencil endpoints and recurrence locality}
\author{Kyle Singh}
\address{Department of Mathematics, University of Iowa, Iowa City, IA 52242}
\email{kyle-singh@uiowa.edu}
\date{}

\begin{document}

\begin{abstract}
Let \(S\) be the unit forward shift, and let
\[
\algA=\K[x]\langle S,S^{-1}\rangle\big/\bigl(Sx-(x+1)S\bigr)
\]
be the algebra of finite Laurent-shift operators with polynomial coefficients over a characteristic-zero field \(\K\). We classify all \(\mathfrak{sl}_2\)-triples in \(\algA\) with fixed lowering operator \(F=1-S\). Writing \(D=S-1\) and \(X=xS^{-1}\), every completion is uniquely determined by \(\lambda\in\K\) and \(g\in\K[S,S^{-1}]\)
\[
H=2(X+g)D-\lambda,
\qquad
E=(X+g)^2D-\lambda(X+g).
\]
We give an intrinsic recognition and reconstruction from \(H\) and determine the exact extreme shifts of \(H\) and \(E\). The associated monic eigenpolynomials form a \(\Delta\)-Appell sequence. Multiplication by \(x\) has finite lower recurrence bandwidth exactly when \(g\in S^{-1}\K[S]\);  if this is the case, the bandwidth equals the right endpoint of the Cartan stencil. Otherwise, \(H\) and \(E\) remain finite-order, while the degree recurrence has an infinite tail whose eventual signed coefficients form a polynomial recovering the lowest Laurent term of \(g\). Over \(\R\), positive-measure orthogonality occurs exactly for translated monic Charlier systems.

\end{abstract}

\maketitle

\section{Introduction}

Difference operators with polynomial eigenfunctions often carry a hidden Lie-algebraic or Heisenberg-algebraic structure.  Classical work on the generalized Bochner problem and Lie-algebraic discretization uses such structures to organize finite-dimensional polynomial modules and to pass from differential to difference equations while preserving spectral data \cite{Turbiner1992,Turbiner1993,SmirnovTurbiner1995,SmirnovTurbiner1996}.  Finite-difference realizations of the canonical commutation relation were developed systematically by G\'orski and Szmigielski \cite{GorskiSzmigielski1998,GorskiSzmigielski2000}, and CCR-preserving shears and their adapted polynomial bases appear in \cite{ChryssomalakosTurbiner2001,DimakisMullerHoissenStriker1996,LeviTempestaWinternitz2004}.  On the algebraic side, realizations of $\mathfrak{sl}_2$ in the first Weyl algebra and finite-dimensional Lie subalgebras of that algebra have been studied under substantially broader equivalence relations \cite{Dixmier1968,Joseph1974,RauschSlupinskiTanasa2006}. Here, we ask a more rigid inverse question.  Let
\[
 (Sf)(x)=f(x+1),
 \qquad
 \Delta=S-1,
 \qquad
 F=-\Delta=1-S,
\]
and let $\algA$ be the algebra of finite Laurent-shift operators with polynomial coefficients.  We prescribe the actual operator $F$, not merely its conjugacy class, and ask the following; Which $E,H\in\algA$ complete $F=1-S$ to an $\mathfrak{sl}_2$-triple? This distinction is important since under unrestricted Weyl-algebra automorphisms, a realization may be simplified while its concrete shift stencil is lost.  Here, the powers of $S$ occurring in $H$ and $E$ are part of the data.  Thus, we classify the fiber over the fixed lowering operator inside a localized difference-operator algebra, rather than classify all embeddings up to automorphism. Set
\[
 D=S-1,
 \qquad
 X=xS^{-1}.
\]
Then $[D,X]=1$.  The fixed-lowering condition forces every triple to arise from the standard one by the shear
\[
 X\longmapsto X+g(D),
 \qquad
 g\in\K[D,(1+D)^{-1}]=\K[S,S^{-1}].
\]
The resulting normal form is
\begin{equation}\label{eq:intro-normal-form}
 H=2(X+g)D-\lambda,
 \qquad
 E=(X+g)^2D-\lambda(X+g),
\end{equation}
with unique $(g,\lambda)$. by doing this, one obtains more than a parametrization;  subtracting the universal $x$-dependent part of $H$ gives a necessary and sufficient recognition test and reconstructs $g$, $\lambda$, and $E$ explicitly. The central phenomenon is a sharp separation between locality in the lattice variable and locality in the degree variable, that is to say finite Laurent-shift support does not imply finite recurrence bandwidth in degree. 
The extreme Laurent exponents of $g$ determine the exact left and right endpoints of the lattice operators $H$ and $E$.  By contrast, multiplication by $x$ in the polynomial eigenbasis has finite lower bandwidth only when no power below $S^{-1}$ occurs.  In that one-sided region the two filtrations coincide numerically
\begin{equation}\label{eq:intro-stencil-bandwidth}
 r=\nu_+(H),
 \qquad
 \nu_+(E)=2r-1,
\end{equation}
where $r$ is the exact lower recurrence bandwidth.  If $g$ contains a power below $S^{-1}$, the lattice operators remain finite-order, but the recurrence develops an infinite tail.  The principal new feature of the complementary region is that this tail is not arbitrary since, after removal of an alternating sign, its eventual coefficients are polynomial in the lag, and that polynomial recovers the lowest Laurent exponent and coefficient of $g$.

This two-sided region is not an optional extension of the one-sided one.  Once $F=1-S$ is prescribed as an actual element of $\algA$, the classification returns the full Laurent ring $\K[S,S^{-1}]$; one does not get to restrict to $S^{-1}\K[S]$ by hypothesis. The region below $S^{-1}$ is therefore exactly the part of the fixed-lowering answer that is not retained by classifications formulated only up to automorphism, and Theorem~\ref{thm:tail} shows that this region has a rigid degree-side signature. For instance, $g(S)=cS^{-r}$ with $r\ge2$ produces a perfectly ordinary $\Delta$-Appell sequence with closed generating function
\[
 \sum_{n\ge0}p_n(x)\frac{z^n}{n!}
 =\exp\Bigl(\frac{c}{1-r}\bigl((1+z)^{1-r}-1\bigr)\Bigr)(1+z)^x,
\]
for which $H$ and $E$ are finite-order difference operators with $\nu_-(H)=-r$ and $\nu_+(H)=0$, while the degree recurrence has a nonzero coefficient at every lag; see Example~\ref{ex:negative-power}.  Families of this kind are legitimate members of the fixed fiber over $F$ but lie outside the finite-band setting treated by the classifications cited above. This construction produces explicit scalar difference equations.  If $g(S)=\sum_jc_jS^j$, then the monic eigenpolynomials satisfy
\begin{equation}\label{eq:intro-difference-equation}
 x\bigl(p_n(x)-p_n(x-1)\bigr)
 +\sum_j c_j\bigl(p_n(x+j+1)-p_n(x+j)\bigr)
 =np_n(x),
\end{equation}
independently of $\lambda$.  For example, the nonclassical choice $g(S)=cS$ gives
\[
 x\bigl(p_n(x)-p_n(x-1)\bigr)
 +c\bigl(p_n(x+2)-p_n(x+1)\bigr)=np_n(x)
\]
and a four-term degree recurrence of exact lower bandwidth two; see Example~\ref{ex:quadratic-band}.

The finite-band $\Delta$-Appell families themselves are classical.  Ben Cheikh and Zaghouani characterize the $d$-orthogonal $\Delta_\omega$-Appell sequences and derive their generating functions, recurrences, and difference equations \cite[Theorems~1.1, 2.1, and 2.2]{BenCheikhZaghouani2003}.  Vinet and Zhedanov derive the $d$-orthogonal Charlier recurrence and difference equations from Heisenberg--Weyl automorphisms \cite[Theorem~1]{VinetZhedanov2009}; broader $d$-orthogonal Sheffer families are studied in \cite{ChaggaraMbarki2018}, and $d$-orthogonal sets arising from $\mathfrak{su}(2)$ representations in \cite{GenestVinetZhedanov2012}.  We therefore do not claim a new classification of the finite-band polynomial sequences.

Closest in method is the automorphism approach of Horozov \cite{Horozov2016,Horozov2016cont}, who applies $\sigma=\exp(\ad_{P(\Delta)})$ with $P$ a polynomial without constant term to move the falling-factorial system into $\Delta$-Appell eigenfunction families, obtaining their generating functions and $(d+2)$-term recurrences in the bispectral framework of \cite{VinetZhedanov2009}; the discrete construction \cite{Horozov2016} has a differential companion \cite{Horozov2016cont} in the Weyl and enveloping-$\mathfrak{sl}_2$ algebras.  The resulting sequences coincide with the eigenbases on our one-sided locus $g\in S^{-1}\K[S]$.  The present work differs in two structural respects.  First, the Horozov constructions produce finite $(d+2)$-term recurrences throughout, so the polynomial sequences always have finite lower bandwidth; the entire two-sided region of Theorem~\ref{thm:tail}, on which the lattice operators stay finite-order while the degree recurrence acquires a genuinely infinite tail, lies outside that category.  Second, the automorphism and bispectral formulations identify realizations up to algebra isomorphism and therefore do not retain the concrete Laurent-shift endpoints of Theorem~\ref{thm:support} or the fixed-operator normalization of Theorem~\ref{thm:classification}, in which $F=1-S$ is prescribed as an actual element of $\algA$. The finite-band overlap is thus a special case of the present fixed-lowering family, while the support-sensitive statements below depend on retaining the concrete Laurent-shift realization rather than passing only to its equivalence class. 

The new contribution is support-sensitive and relative to the fixed
operator $F=1-S$.  The point of fixing $F$ is not merely a choice of
normal form; classifications up to Weyl- or Heisenberg-algebra
automorphism may identify realizations whose concrete Laurent-shift
stencils are different, and therefore do not retain the endpoint data
studied here. Likewise, the finite-band $\Delta$-Appell and $d$-orthogonal
classifications cited above describe polynomial sequences on the
one-sided locus, but do not formulate the problem of classifying
completions of the prescribed operator $F=1-S$ inside the finite
Laurent-shift algebra. In particular, those finite-band results do not
cover the two-sided Laurent region considered here, in which $H$ and
$E$ remain finite-order difference operators while multiplication by
$x$ has an infinite degree-side recurrence. The results below therefore concern the fixed fiber over $F$, together with support invariants that are not retained after passing only to an unrestricted equivalence class. 

CCR-preserving shears are familiar, but Theorem~\ref{thm:classification} proves that every completion of the prescribed $F$ inside the finite Laurent-shift algebra is obtained from one unique Laurent shear.  The operator $F$ acts locally nilpotently on $\K[x]$, but it is retained as an actual element of $\algA$, rather than only up to automorphism.  Proposition~\ref{prop:recognition} reconstructs $g$, $\lambda$, and $E$ from the Cartan operator alone.  Theorem~\ref{thm:support} determines the exact extreme shifts of $H$ and $E$, including the boundary cases in which cancellation might otherwise be expected. Theorems~\ref{thm:recurrence} and \ref{thm:numerical-bandwidth} identify the classical finite-band $\Delta$-Appell locus intrinsically inside the fixed-lowering parameter space: it is exactly the one-sided region $g\in S^{-1}\K[S]$, and on that region the endpoint identities \eqref{eq:intro-stencil-bandwidth} hold.
Theorem~\ref{thm:tail} treats the complementary region, which is absent from the finite-band classifications.  Although $H$ and $E$ retain finite lattice support, the degree recurrence has an infinite tail whose eventual polynomial profile recovers the lowest Laurent term of $g$. Theorem~\ref{thm:positive} identifies the exact intersection of this fixed-lowering finite-shift category with positive-measure orthogonality; the resulting systems are precisely translated monic Charlier polynomials.

The scalar $\lambda$ controls the $\mathfrak{sl}_2$ weights but not the polynomial system or either locality invariant.  Section~2 proves the fixed-lowering normal form and distinguishes algebra automorphisms from similarity on the polynomial module.  Section~3 establishes exact stencil endpoints and reconstruction from $H$.  Section~4 constructs and recognizes the polynomial eigenbases and their scalar difference equations.  Section~5 proves the recurrence formula, the stencil--bandwidth identities, and the negative-Laurent tail invariant before relating the finite-band locus to $d$-orthogonality.  Section~6 determines the positive-measure orthogonality locus, and Section~7 gives explicit examples.

\section{The Laurent shift algebra}

Throughout, $\K$ is a field of characteristic zero.  Let $S$ denote the unit forward shift.  The algebra of finite-order difference operators with polynomial coefficients is
\begin{equation}\label{eq:shift-algebra}
 \algA=\K[x]\langle S,S^{-1}\rangle\big/\bigl(Sx-(x+1)S\bigr).
\end{equation}
Thus an element of $\algA$ is a finite sum $\sum_j t_j(x)S^j$ with $t_j(x)\in\K[x]$; no completion in the shift variable is taken.  This finite-order requirement is essential below.  By contrast, after expanding at $D=0$, a Laurent polynomial in $S=1+D$ can have an infinite Taylor series in $D$.

Set
\[
 D=S-1,
 \qquad
 X=xS^{-1},
 \qquad
 R=\K[D,(D+1)^{-1}]=\K[S,S^{-1}].
\]
As a subring of $\K(D)$, the ring $R$ consists of rational functions whose only possible pole is at $D=-1$; in particular, every element of $R$ is regular at $D=0$.  A prime on an element of $R$ denotes formal differentiation with respect to $D$.

\begin{proposition}\label{prop:ore}
The operators $D$ and $X$ satisfy $[D,X]=1$.  Moreover, $\algA$ is isomorphic to the Ore extension $R[X;-\partial_D]$, in which
\[
 Xr=rX-r',\qquad r\in R.
\]
Consequently every $T\in\algA$ has a unique normal form
\[
 T=\sum_{j=0}^{N}X^j r_j(D),
 \qquad
 r_j(D)\in R.
\]
\end{proposition}

\begin{proof}
The defining relation gives
\[
 DX=(S-1)xS^{-1}=(x+1)-xS^{-1},
 \qquad
 XD=xS^{-1}(S-1)=x-xS^{-1},
\]
so $[D,X]=1$.  Equivalently, $[X,D]=-1$, and hence first for polynomials and then for all $r\in R$,
\[
 r(D)X=Xr(D)+r'(D),
 \qquad
 Xr(D)=r(D)X-r'(D).
\]
Conversely, in the Ore extension set $S=D+1$ and $x=X(D+1)$.  Then
\[
 Sx=(D+1)X(D+1)=(X(D+1)+1)(D+1)=(x+1)S,
\]
so the defining relation of $\algA$ holds.  The assignments
\[
 D=S-1,
 \quad X=xS^{-1},
 \qquad
 S=D+1,
 \quad x=X(D+1)
\]
are inverse homomorphisms.  The normal form is the standard PBW normal form for an Ore extension; see \cite{Ore1933,McConnellRobson2001}.
\end{proof}

\begin{lemma}\label{lem:faithful}
The natural action of $\algA$ on $\calP=\K[x]$ is faithful.  The locally finite action of $\K[[D]]$ on $\calP$ is faithful as well.
\end{lemma}

\begin{proof}
Let $T=\sum_{j=a}^{b}t_j(x)S^j\in\algA$ annihilate every polynomial.  Applying $T$ coefficientwise to the formal exponential $e^{tx}$ gives
\[
 0=T(e^{tx})=e^{tx}\sum_{j=a}^{b}t_j(x)e^{jt}
 \quad\text{in }\K[x][[t]].
\]
Taking the first $b-a+1$ derivatives with respect to $t$ at $t=0$ yields a Vandermonde system for the finitely many polynomials $t_j(x)$; hence every $t_j$ is zero.  Thus the action of $\algA$ is faithful.

If $A(D)=\sum_{k\ge0}a_kD^k$ acts as zero on $\calP$, let $m$ be the least index with $a_m\ne0$.  Applying $A(D)$ to $x^{\underline m}$ gives the nonzero constant $a_m m!$, a contradiction.  Hence the locally finite $\K[[D]]$-action is faithful.
\end{proof}

We use the convention that an $\mathfrak{sl}_2$-triple $(E,H,F)$ satisfies
\begin{equation}\label{eq:sl2-relations}
 [H,E]=2E,
 \qquad
 [H,F]=-2F,
 \qquad
 [E,F]=H.
\end{equation}

\begin{lemma}\label{lem:evaluation}
The evaluation homomorphism
\[
 \ev_0:R=\K[D,(D+1)^{-1}]\longrightarrow\K,
 \qquad
 r\longmapsto r(0),
\]
has kernel $DR$.
\end{lemma}

\begin{proof}
Write $r(D)=a(D)/(D+1)^m$ with $a(D)\in\K[D]$ and $m\ge0$.  Since the denominator has value $1$ at $D=0$, one has $r(0)=0$ if and only if $a(0)=0$.  This is equivalent to $a(D)=Db(D)$ for some $b(D)\in\K[D]$, and hence to $r(D)\in DR$.
\end{proof}

\begin{theorem}\label{thm:classification}
Let $E,H\in\algA$ and fix $F=-D=1-S$.  Then $(E,H,F)$ is an $\mathfrak{sl}_2$-triple if and only if there exist unique $\lambda\in\K$ and $g(D)\in R$ such that, with
\[
 Y=X+g(D),
\]
one has
\begin{equation}\label{eq:normal-form}
 H=2YD-\lambda,
 \qquad
 E=Y^2D-\lambda Y.
\end{equation}
For every $g\in R$, the assignment
\begin{equation}\label{eq:shear-automorphism}
 \sigma_g(D)=D,
 \qquad
 \sigma_g(X)=X+g(D)
\end{equation}
defines a $\K$-algebra automorphism of $\algA$ with inverse $\sigma_{-g}$.  If
\[
 H^{(0)}_\lambda=2XD-\lambda,
 \qquad
 E^{(0)}_\lambda=X^2D-\lambda X,
\]
then every fixed-lowering triple is uniquely
\[
 (E,H,F)=\bigl(\sigma_g(E^{(0)}_\lambda),
 \sigma_g(H^{(0)}_\lambda),-D\bigr).
\]
\end{theorem}

\begin{proof}
Suppose $(E,H,F)$ satisfies \eqref{eq:sl2-relations}.  Write
\[
 H=\sum_{j=0}^{N}X^j h_j(D)
\]
in the normal form of Proposition~\ref{prop:ore}.  Since $F=-D$, the relation $[H,F]=-2F$ is equivalent to $[H,D]=-2D$.  Because $[X^j,D]=-jX^{j-1}$, comparison of normal forms gives
\[
 h_1(D)=2D,
 \qquad
 h_j(D)=0\quad(j\ge2).
\]
Thus $H=2XD+\phi(D)$ for a unique $\phi\in R$. Similarly, writing $E=\sum_{j=0}^{M}X^j e_j(D)$, the relation $[E,F]=H$, or $-[E,D]=H$, forces
\[
 E=X^2D+X\phi(D)+\psi(D)
\]
for a unique $\psi\in R$.  For readability, we record the normal-order calculation.  The identities
\[
 [XD,X]=X,
 \qquad
 [XD,D]=-D,
 \qquad
 [\phi(D),X]=\phi'(D),
 \qquad
 [XD,r(D)]=-Dr'(D)
\]
give
\[
\begin{aligned}
 [2XD,X^2D]&=2X^2D,\\
 [2XD,X\phi]&=2X\phi-2XD\phi',\\
 [2XD,\psi]&=-2D\psi',\\
 [\phi,X^2D]&=2X\phi'D+D\phi'',\\
 [\phi,X\phi]&=\phi\phi'.
\end{aligned}
\]
Since $D$ commutes with $\phi'$, the two terms containing $X D\phi'$ cancel.  Hence
\[
 [H,E]=2X^2D+2X\phi-2D\psi'+D\phi''+\phi\phi'.
\]
Therefore, $[H,E]=2E$ is equivalent to
\begin{equation}\label{eq:psi-ode}
 2D\psi'(D)+2\psi(D)=D\phi''(D)+\phi(D)\phi'(D).
\end{equation}
The integration step is internal to $R$.  Indeed
\[
 (2D\psi)'=2D\psi'+2\psi,
\]
whereas
\[
 \left(D\phi'-\phi+\frac12\phi^2\right)'
 =D\phi''+\phi\phi'.
\]
Thus the two sides of \eqref{eq:psi-ode} are total derivatives.  Because $R\subset\K(D)$ and $\K$ has characteristic zero, the kernel of $\partial_D:R\to R$ is exactly $\K$.  Hence
\[
 2D\psi=D\phi'-\phi+\frac12\phi^2+C
\]
for some $C\in\K$.  Evaluating at $D=0$ gives
\[
 C=\phi(0)-\frac12\phi(0)^2.
\]
Set
\begin{equation}\label{eq:recover-lambda-g}
 \lambda=-\phi(0),
 \qquad
 g(D)=\frac{\phi(D)-\phi(0)}{2D}.
\end{equation}
By Lemma~\ref{lem:evaluation}, the numerator in the definition of $g$ lies in $DR$, so $g\in R$.  Thus
\[
 \phi(D)=2Dg(D)-\lambda.
\]
Substituting this identity into the integrated equation gives, without division by any element outside $R$
\[
 2D\psi=2D^2g'+2D^2g^2-2\lambda Dg,
\]
and hence
\[
 \psi(D)=Dg'(D)+Dg(D)^2-\lambda g(D).
\]
Since $g(D)$ commutes with $D$,
\[
 (X+g)^2D-\lambda(X+g)
 =X^2D+X(2Dg-\lambda)+Dg'+Dg^2-\lambda g=E,
\]
and $H=2(X+g)D-\lambda$.  The formulas \eqref{eq:recover-lambda-g} show uniqueness. Conversely, if $Y=X+g(D)$, then $[D,Y]=1$, so $[YD,Y]=Y$ and $[YD,D]=-D$.  Therefore
\[
 [2YD,Y^2D]=2Y^2D,
 \qquad
 [2YD,-\lambda Y]=-2\lambda Y,
\]
and hence
\[
 [2YD-\lambda,Y^2D-\lambda Y]=2(Y^2D-\lambda Y).
\]
Similarly
\[
 [2YD-\lambda,-D]=-2(-D),
 \qquad
 [Y^2D-\lambda Y,-D]=2YD-\lambda.
\]
Thus \eqref{eq:normal-form} defines an $\mathfrak{sl}_2$-triple.  Finally, the images in \eqref{eq:shear-automorphism} satisfy
\[
 [\sigma_g(D),\sigma_g(X)]=[D,X+g(D)]=1,
\]
so Proposition~\ref{prop:ore} gives an endomorphism $\sigma_g$ of $\algA$.  The composition laws $\sigma_g\sigma_h=\sigma_{g+h}$ and $\sigma_0=\id$ show that its inverse is $\sigma_{-g}$, and applying $\sigma_g$ to the standard triple gives \eqref{eq:normal-form}.
\end{proof}

\begin{remark}\label{rem:relative-category}
Theorem~\ref{thm:classification} classifies the fiber over the actual operator $F=1-S$.  No quotient by automorphisms is taken.  This is why the Laurent support of $H$ and $E$ remains meaningful.  The automorphisms $\sigma_g$ describe the fiber, but applying an arbitrary automorphism of a larger Weyl algebra would generally change the prescribed lowering operator and erase the stencil data retained below.
\end{remark}

\begin{remark}\label{rem:lambda}
The scalar $\lambda$ is independent of the stencil and recurrence parameters studied below.  Once $g$ is fixed, changing $\lambda$ shifts the $H$-eigenvalues and changes the raising action of $E$, hence the $\mathfrak{sl}_2$-module parameter, but it does not change the conjugating operator, polynomial eigenbasis, generating function, recurrence, stencil bounds, or orthogonality locus.
\end{remark}

\section{Shift support and endpoint rigidity}

For a nonzero operator
\[
 T=\sum_{j=a}^{b}t_j(x)S^j,
 \qquad
 t_a,t_b\ne0,
\]
define
\[
 \supp_S(T)=\{j\in\Z:t_j\ne0\},
 \qquad
 \nu_-(T)=a,
 \qquad
 \nu_+(T)=b.
\]

\begin{theorem}\label{thm:support}
Let $(E,H,F)$ be associated with $(g,\lambda)$ as in Theorem~\ref{thm:classification}.  If $g=0$, then
\[
 \nu_-(H)=-1,
 \quad \nu_+(H)=0,
 \qquad
 \nu_-(E)=-2,
 \quad \nu_+(E)=-1.
\]
Suppose $g\ne0$ and write
\begin{equation}\label{eq:g-pq}
 g(S)=\sum_{j=p}^{q}c_jS^j,
 \qquad
 c_pc_q\ne0.
\end{equation}
Then
\begin{equation}\label{eq:H-endpoints}
 \nu_-(H)=\min\{-1,p\},
 \qquad
 \nu_+(H)=\max\{0,q+1\},
\end{equation}
and
\begin{equation}\label{eq:E-endpoints}
 \nu_-(E)=\min\{-2,2p\},
 \qquad
 \nu_+(E)=\max\{-1,2q+1\}.
\end{equation}
In particular, these endpoints are independent of $\lambda$ and there are no exceptional coefficient loci on which an extreme shift disappears.
\end{theorem}

\begin{proof}
Since $D=S-1$ and $X=xS^{-1}$,
\[
 XD=x-xS^{-1},
\]
so
\begin{equation}\label{eq:H-shift-form}
 H=2x-2xS^{-1}+2g(S)(S-1)-\lambda.
\end{equation}
The constant-coefficient operator $g(S)(S-1)$ has lowest shift $p$, with coefficient $-c_p$, and highest shift $q+1$, with coefficient $c_q$.  If either endpoint coincides with $-1$ or $0$, the $x$-dependent coefficient from $2x-2xS^{-1}$ prevents complete cancellation.  This proves \eqref{eq:H-endpoints}. For $E$ use
\begin{equation}\label{eq:E-expand}
 E=X^2D+2XgD+g'D+g^2D-\lambda X-\lambda g.
\end{equation}
The standard term is
\[
 X^2D=x(x-1)(S^{-1}-S^{-2}),
\]
with extreme shifts $-2$ and $-1$.  The term $g^2D$ has lowest shift $2p$, with coefficient $-c_p^2$, and highest shift $2q+1$, with coefficient $c_q^2$.  The terms $XgD$ and $g'D$ are supported between $p-1$ and $q$, while $\lambda g$ is supported between $p$ and $q$. Away from the boundary values $p=-1$ and $q=-1$, the source of each extreme shift is unique
\[
\begin{array}{c|c@{\qquad}c|c}
 p<-1 & \text{lowest shift }2p\text{ from }g^2D
 & p>-1 & \text{lowest shift }-2\text{ from }X^2D\\
 q>-1 & \text{highest shift }2q+1\text{ from }g^2D
 & q<-1 & \text{highest shift }-1\text{ from }X^2D.
\end{array}
\]
The remaining terms in \eqref{eq:E-expand} lie strictly between the displayed endpoints in these cases.  It remains only to exclude cancellation at the two boundary values.  If $p=-1$, the coefficient at shift $-2$ is
\[
 -x(x-1)-2c_{-1}x+c_{-1}-c_{-1}^2
 =-(x+c_{-1})(x+c_{-1}-1),
\]
a nonzero quadratic polynomial.  If $q=-1$, the coefficient at shift $-1$ is
\[
 x(x-1)+2c_{-1}x-c_{-1}+c_{-1}^2-\lambda x-\lambda c_{-1}
 =(x+c_{-1})(x+c_{-1}-1-\lambda),
\]
again nonzero.  The two boundary coefficients are therefore nonzero, and the uniqueness discussion above proves \eqref{eq:E-endpoints}.  This also isolates the rigidity mechanism; away from a boundary an extreme shift has a unique algebraic source, while at a boundary its coefficient has a nonzero highest-degree term in $x$ and cannot be canceled by constant-coefficient contributions.
\end{proof}

\begin{definition}\label{def:layers}
For $m\ge0$, a fixed-lowering triple belongs to the $m$th one-sided stencil layer $\calL_m$ if
\begin{equation}\label{eq:layer}
 \supp_S(H)\subseteq[-1,m]\cap\Z,
 \qquad
 \supp_S(E)\subseteq[-2,2m-1]\cap\Z.
\end{equation}
The word \emph{layer} refers only to these nested support bounds; no quotient or geometric stratification is intended.
\end{definition}

\begin{corollary}\label{cor:all-layers}
For every $m\ge0$, a fixed-lowering triple belongs to $\calL_m$ if and only if
\begin{equation}\label{eq:layer-g-space}
 g(S)\in \Span_{\K}\{S^{-1},1,S,\ldots,S^{m-1}\},
\end{equation}
where for $m=0$ the right-hand side means $\K S^{-1}$.
\end{corollary}

\begin{proof}
The case $g=0$ is immediate.  For $g\ne0$, Theorem~\ref{thm:support} shows that the lower support bounds in \eqref{eq:layer} are equivalent to $p\ge-1$, while the upper support bounds are equivalent to $q\le m-1$.  Together these conditions are precisely \eqref{eq:layer-g-space}.
\end{proof}

\begin{corollary}\label{cor:minimal-layer}
A fixed-lowering triple belongs to $\calL_0$ if and only if
\[
 g(S)=cS^{-1}
\]
for some $c\in\K$.  In that case
\[
 X+g(S)=(x+c)S^{-1},
\]
so the triple is obtained from the standard realization by the lattice translation $x\mapsto x+c$.
\end{corollary}

\begin{proof}
This is Corollary~\ref{cor:all-layers} with $m=0$.  The last identity identifies the shear with translation.
\end{proof}

\begin{corollary}\label{cor:first-layer}
A fixed-lowering triple belongs to $\calL_1$ if and only if
\[
 g(S)=c_{-1}S^{-1}+c_0
\]
for some $c_{-1},c_0\in\K$.  Equivalently
\[
 \supp_S(H)\subseteq\{-1,0,1\},
 \qquad
 \supp_S(E)\subseteq\{-2,-1,0,1\}.
\]
\end{corollary}

\begin{proposition}\label{prop:recognition}
Let $H\in\algA$ and define
\[
 R_H=H-(2x-2xS^{-1}).
\]
There exists $E\in\algA$ such that $(E,H,1-S)$ is an $\mathfrak{sl}_2$-triple if and only if
\[
 R_H\in\K[S,S^{-1}].
\]
Equivalently, recognition and reconstruction consist of the following three steps: (i) subtract the universal $x$-dependent part $2x-2xS^{-1}$ (ii) test whether the remainder has constant coefficients (iii) when it does, evaluate at $S=1$ and divide its vanishing part by $2(S-1)$.  Explicitly, the parameters and the raising operator are recovered uniquely by
\begin{equation}\label{eq:recognition}
 \lambda=-R_H(1),
 \qquad
 g(S)=\frac{R_H(S)-R_H(1)}{2(S-1)},
\end{equation}
\[
 E=\bigl(xS^{-1}+g(S)\bigr)^2(S-1)
 -\lambda\bigl(xS^{-1}+g(S)\bigr).
\]
\end{proposition}

\begin{proof}
Equation \eqref{eq:H-shift-form} gives
\[
 R_H=2g(S)(S-1)-\lambda.
\]
Thus $R_H$ is a Laurent polynomial, $R_H(1)=-\lambda$, and subtraction gives \eqref{eq:recognition}.  Conversely, evaluation at $S=1$ defines a homomorphism $\K[S,S^{-1}]\to\K$ with kernel $(S-1)\K[S,S^{-1}]$.  Hence, if $R_H\in\K[S,S^{-1}]$, the difference $R_H(S)-R_H(1)$ is divisible by $S-1$ in the Laurent polynomial ring.  The formulas therefore define $g\in\K[S,S^{-1}]$ and $\lambda\in\K$, and Theorem~\ref{thm:classification} supplies the unique $E$.
\end{proof}

\section{Polynomial eigenbases and classified difference equations}

Let $\calP=\K[x]$ and write
\[
 x^{\underline n}=x(x-1)\cdots(x-n+1),
 \qquad
 x^{\underline0}=1.
\]
Then
\[
 D x^{\underline n}=n x^{\underline{n-1}},
 \qquad
 Xx^{\underline n}=x^{\underline{n+1}}.
\]
For $g\in R$, let $G(D)\in D\K[[D]]$ be the unique formal power series satisfying $G'(D)=g(D)$.  Since $D$ is locally nilpotent on $\calP$, the operator
\[
 U_g=\exp(G(D))
\]
is well defined and invertible on every finite-dimensional polynomial subspace.  All generating-function identities below are identities in $\K[x][[z]]$.  The operator $U_g$ is an automorphism of $\calP$; it need not be an element of the finite-order algebra $\algA$.

\begin{theorem}\label{thm:eigenbasis}
Let $(E,H,F)$ correspond to $(g,\lambda)$ as in Theorem~\ref{thm:classification}.  On $\calP$,
\[
 Y=U_gXU_g^{-1},
 \qquad
 F=U_g(-D)U_g^{-1},
\]
and consequently
\[
 H=U_g(2XD-\lambda)U_g^{-1},
 \qquad
 E=U_g(X^2D-\lambda X)U_g^{-1}.
\]
The polynomials
\[
 p_n(x)=U_gx^{\underline n},
 \qquad n\ge0,
\]
are monic of degree $n$ and satisfy
\begin{equation}\label{eq:basis-actions}
 Dp_n=np_{n-1},
 \qquad
 Yp_n=p_{n+1},
\end{equation}
\begin{equation}\label{eq:sl2-actions}
 Fp_n=-np_{n-1},
 \qquad
 Hp_n=(2n-\lambda)p_n,
 \qquad
 Ep_n=(n-\lambda)p_{n+1}.
\end{equation}
For each $n$, the polynomial $p_n$ is the unique monic degree-$n$ eigenpolynomial of $H$ with eigenvalue $2n-\lambda$.  Their exponential generating function is
\begin{equation}\label{eq:egf}
 \sum_{n=0}^{\infty}p_n(x)\frac{z^n}{n!}
 =e^{G(z)}(1+z)^x.
\end{equation}
\end{theorem}

\begin{proof}
Since $[G(D),X]=G'(D)=g(D)$ and the next iterated commutator vanishes, the Baker--Campbell--Hausdorff formula gives
\[
 U_gXU_g^{-1}=X+g(D)=Y.
\]
The operator $U_g$ commutes with $D$, proving the similarity formulas.  Applying them to the falling-factorial basis gives \eqref{eq:basis-actions} and \eqref{eq:sl2-actions}.  Since $U_g$ is the identity plus degree-lowering terms, $p_n$ is monic of degree $n$.  The eigenspaces of $2XD-\lambda$ on $\calP$ are the one-dimensional spans of $x^{\underline n}$; similarity by $U_g$ therefore proves uniqueness of the monic eigenpolynomials.  Finally,
\[
 \sum_{n\ge0}x^{\underline n}\frac{z^n}{n!}=(1+z)^x,
\]
and $D$ acts on $(1+z)^x$ by multiplication by $z$.  Hence $U_g$ acts on this generating function by multiplication by $e^{G(z)}$.
\end{proof}

\begin{remark}\label{rem:shear-similarity}
The automorphism $\sigma_g$ in Theorem~\ref{thm:classification} belongs to the algebraic classification inside $\algA$.  The operator $U_g=\exp(G(D))$ instead implements the same shear by similarity on the polynomial module $\calP$; unless $G$ is suitably special, $U_g$ is an infinite formal series and does not belong to the finite-order algebra $\algA$.
\end{remark}

\begin{corollary}\label{cor:explicit-equation}
Suppose
\[
 g(S)=\sum_{j=p}^{q}c_jS^j.
\]
Then every eigenpolynomial $p_n$ in Theorem~\ref{thm:eigenbasis} satisfies
\begin{equation}\label{eq:general-difference-equation}
 x\bigl(p_n(x)-p_n(x-1)\bigr)
 +\sum_{j=p}^{q}c_j\bigl(p_n(x+j+1)-p_n(x+j)\bigr)
 =np_n(x).
\end{equation}
The spectral equation is independent of $\lambda$.  Conversely, every polynomial eigenbasis arising from a fixed-lowering triple satisfies \eqref{eq:general-difference-equation} for the unique Laurent polynomial recovered from $H$ by Proposition~\ref{prop:recognition}.
\end{corollary}

\begin{proof}
From $Hp_n=(2n-\lambda)p_n$ and $H=2(X+g)D-\lambda$ one obtains
\[
 (X+g(S))(S-1)p_n=np_n.
\]
Now
\[
 X(S-1)=xS^{-1}(S-1)=x(1-S^{-1}),
\]
and
\[
 g(S)(S-1)=\sum_{j=p}^{q}c_j(S^{j+1}-S^j).
\]
Evaluating these shifts at $x$ gives \eqref{eq:general-difference-equation}.
\end{proof}

The preceding similarity belongs to the general finite-operator-calculus and CCR-preserving framework \cite{ChryssomalakosTurbiner2001,Roman1984,RotaKahanerOdlyzko1973}.  The general $\Delta$-Appell generating-function form also appears in \cite[Theorem~2.1]{BenCheikhZaghouani2003}.  The next theorem isolates the additional restriction imposed by finite Laurent-shift support in that the logarithmic derivative may have no pole except possibly at $z=-1$.

\begin{theorem}\label{thm:recognize-sequence}
Let $\{p_n\}_{n\ge0}$ be a monic polynomial sequence satisfying
\[
 Dp_n=np_{n-1}.
\]
Write its exponential generating function uniquely as
\begin{equation}\label{eq:general-delta-appell}
 \sum_{n=0}^{\infty}p_n(x)\frac{z^n}{n!}
 =A(z)(1+z)^x,
 \qquad A(0)=1.
\end{equation}
Then the following are equivalent.
\begin{enumerate}[label=\textup{(\roman*)},leftmargin=2.3em]
\item There exist $E,H\in\algA$ and $\lambda\in\K$ such that $(E,H,1-S)$ is an $\mathfrak{sl}_2$-triple and $Hp_n=(2n-\lambda)p_n$.
\item The logarithmic derivative satisfies
\[
 \frac{A'(z)}{A(z)}\in\K[z,(1+z)^{-1}].
\]
\end{enumerate}
Here $\K[z,(1+z)^{-1}]$ is regarded as a subring of $\K[[z]]$
via Taylor expansion at $z=0$. When these conditions hold, the shear is uniquely recovered by
\begin{equation}\label{eq:g-from-A}
 g(S)=\left.\frac{A'(z)}{A(z)}\right|_{z=S-1},
\end{equation}
whereas $\lambda\in\K$ is arbitrary.  If $g$ is written as in \eqref{eq:g-pq}, then
\begin{equation}\label{eq:A-explicit}
 A(z)=(1+z)^{c_{-1}}
 \exp\!\left(
 \sum_{\substack{j=p\\j\ne-1}}^{q}
 \frac{c_j}{j+1}\bigl((1+z)^{j+1}-1\bigr)
 \right),
\end{equation}
interpreted in $\K[[z]]$, with
\[
 (1+z)^\alpha:=\exp\!\bigl(\alpha\log(1+z)\bigr)
 \qquad(\alpha\in\K).
\]
\end{theorem}

\begin{proof}
To justify \eqref{eq:general-delta-appell} directly, set $P(x,z)=\sum_{n\ge0}p_n(x)z^n/n!$.  The lowering relation gives $DP=zP$, hence $P(x+1,z)=(1+z)P(x,z)$.  Therefore $P(x,z)/(1+z)^x$ is independent of $x$; monicity gives a unique series $A(z)$ with $A(0)=1$.  Equivalently, if $U=A(D)$ is interpreted as a formal power series in the locally nilpotent operator $D$, then $p_n=Ux^{\underline n}$.  Since $A(0)=1$, the formal logarithm $G=\log A$ exists and
\[
 UXU^{-1}=X+G'(D)=X+\frac{A'(D)}{A(D)}.
\]
Assume (i).  By Theorem~\ref{thm:classification}, write $H=2YD-\lambda$ with $Y=X+g(D)$.  Applying the eigenvalue equation to $p_{n+1}$ gives
\[
 2(n+1)Yp_n-\lambda p_{n+1}
 =(2(n+1)-\lambda)p_{n+1},
\]
so $Yp_n=p_{n+1}$ for every $n\ge0$.  Hence $Y=UXU^{-1}$ on the basis $\{p_n\}$.  The second assertion of Lemma~\ref{lem:faithful} identifies the corresponding locally finite power series in $D$, so
\[
 g(D)=\frac{A'(D)}{A(D)}\in R.
\]
This is (ii).  Conversely, if (ii) holds, define $g(D)=A'(D)/A(D)\in R$.  For any $\lambda\in\K$, Theorem~\ref{thm:classification} defines a triple, and Theorem~\ref{thm:eigenbasis} gives the required eigenvalue equation.  Formula \eqref{eq:A-explicit} follows by integrating
\[
 G'(z)=g(1+z)=\sum_{j=p}^{q}c_j(1+z)^j
\]
with $G(0)=0$, and exponentiating.
\end{proof}

\section{Degree-side recurrence and the stencil--bandwidth filtration}

Let $\{p_n\}$ be the basis in Theorem~\ref{thm:eigenbasis}.  Define
\begin{equation}\label{eq:h-def}
 h(D)=(1+D)g(D)=Sg(S)
\end{equation}
and write its Taylor expansion at $D=0$ as
\[
 h(D)=\sum_{k=0}^{\infty}h_kD^k.
\]

\begin{lemma}\label{lem:intersection}
Inside $\K(S)$ one has
\[
 \K[S,S^{-1}]\cap\K[D]=\K[S],
 \qquad D=S-1.
\]
Equivalently, a Laurent polynomial in $S$ is a polynomial in $D$ if and only if it has no negative power of $S$.
\end{lemma}

\begin{proof}
Since $\K[D]=\K[S]$, only the intersection statement requires comment.  If a Laurent polynomial $S^{-m}a(S)$ with $a(S)\in\K[S]$ and $a(0)\ne0$ were a polynomial in $S$, then it would be regular at $S=0$, which is impossible unless $m=0$.
\end{proof}

\begin{definition}\label{def:bandwidth}
The basis $\{p_n\}$ has \emph{uniform lower bandwidth} $r$ if $r$ is the least nonnegative integer such that
\[
 xp_n\in\Span\{p_{n+1},p_n,p_{n-1},\ldots,p_{n-r}\}
\]
for every $n$, with $p_j=0$ for $j<0$.  If no such $r$ exists, the lower bandwidth is infinite.
\end{definition}

\begin{theorem}\label{thm:recurrence}
For every $n\ge0$,
\begin{equation}\label{eq:degree-recurrence}
 xp_n=p_{n+1}+(n-h_0)p_n-
 \sum_{k=1}^{n}h_k\fall{n}{k}p_{n-k},
\end{equation}
where $\fall{n}{k}=n(n-1)\cdots(n-k+1)$ denotes the falling factorial of the integer $n$, with $\fall{n}{0}=1$.
Consequently, the lower bandwidth is finite if and only if
\begin{equation}\label{eq:finite-band-condition}
 g(S)\in S^{-1}\K[S].
\end{equation}
More precisely, suppose $g\ne0$ is written as in \eqref{eq:g-pq}.
\begin{enumerate}[label=\textup{(\alph*)},leftmargin=2.3em]
\item If $p<-1$, the lower bandwidth is infinite.
\item If $p\ge-1$ and $q=-1$, the lower bandwidth is zero and
\[
 xp_n=p_{n+1}+(n-c_{-1})p_n.
\]
\item If $p\ge-1$ and $q\ge0$, the lower bandwidth is exactly $q+1$.  The coefficient of the extreme term $p_{n-q-1}$ is
\[
 -c_q\fall{n}{q+1},
\]
which is nonzero for $n\ge q+1$.
\end{enumerate}
\end{theorem}

\begin{proof}
Since $x=XS=X(D+1)$ and $X=Y-g(D)$,
\[
 x=Y(D+1)-g(D)(D+1)=Y(D+1)-h(D).
\]
Using $Dp_n=np_{n-1}$ and $Yp_n=p_{n+1}$ gives
\[
 Y(D+1)p_n=p_{n+1}+np_n.
\]
Moreover
\[
 h(D)p_n=\sum_{k=0}^{n}h_kD^kp_n
 =\sum_{k=0}^{n}h_k\fall{n}{k}p_{n-k},
\]
which proves \eqref{eq:degree-recurrence}. For every $k$ with $h_k\ne0$, the coefficient of $p_{n-k}$ is $-h_k\fall{n}{k}$, nonzero for every $n\ge k$.  Thus finite lower bandwidth is equivalent to $h(D)$ being a polynomial.  Since $h(D)=Sg(S)$ is a Laurent polynomial in $S$, Lemma~\ref{lem:intersection} shows that it is a polynomial in $D=S-1$ if and only if it contains no negative powers of $S$.  This is equivalent to \eqref{eq:finite-band-condition}.  If $q=-1$, then $g(S)=c_{-1}S^{-1}$ and $h(D)=c_{-1}$.  If $q\ge0$, then $h(D)$ has degree $q+1$ and leading coefficient $c_q$, giving the stated extreme coefficient and exact bandwidth.
\end{proof}

\begin{theorem}\label{thm:numerical-bandwidth}
Let $r\in\Z_{\ge0}\cup\{\infty\}$ be the exact lower bandwidth of the eigenbasis.  The following conditions are equivalent:
\begin{enumerate}[label=\textup{(\roman*)},leftmargin=2.3em]
\item $r<\infty$;
\item $g(S)\in S^{-1}\K[S]$;
\item $\nu_-(H)=-1$;
\item $\nu_-(E)=-2$.
\end{enumerate}
When these conditions hold,
\begin{equation}\label{eq:numerical-bandwidth}
 r=\nu_+(H),
 \qquad
 \nu_+(E)=2r-1.
\end{equation}
Thus, within the fixed-lowering Laurent-shift category, the recurrence bandwidth is determined by the right endpoint of the Cartan stencil, and the right endpoint of the raising stencil is then forced.
\end{theorem}

\begin{proof}
The case $g=0$ follows directly from Theorems~\ref{thm:support} and \ref{thm:recurrence}: one has $r=0$, $\nu_+(H)=0$, and $\nu_+(E)=-1$.  Suppose $g\ne0$ and write it as in \eqref{eq:g-pq}.  By Theorem~\ref{thm:recurrence}, finite bandwidth is equivalent to $p\ge-1$, which is equivalent to $g\in S^{-1}\K[S]$.  By Theorem~\ref{thm:support}, the same inequality is equivalent separately to $\nu_-(H)=-1$ and to $\nu_-(E)=-2$.  If $q=-1$, then $r=0$, $\nu_+(H)=0$, and $\nu_+(E)=-1$.  If $q\ge0$, then $r=q+1$, while Theorem~\ref{thm:support} gives $\nu_+(H)=q+1$ and $\nu_+(E)=2q+1=2r-1$.
\end{proof}

For an expansion $v=\sum_m a_mp_m$ in the polynomial basis, write $[p_m]v=a_m$.

\begin{theorem}\label{thm:tail}
Suppose $g\ne0$ is written as in \eqref{eq:g-pq} and $p<-1$.  Put
\[
 J=\min\{q,-2\}.
\]
For every integer $k>\max\{q+1,0\}$, the Taylor coefficient of $h(D)=Sg(S)$ is
\begin{equation}\label{eq:tail-coefficients}
 h_k=(-1)^kP_g(k),
 \qquad
 P_g(k)=\sum_{j=p}^{J}c_j
 \binom{k-j-2}{-j-2}.
\end{equation}
The expression $P_g(t)$ is a nonzero polynomial in $t$ of degree $-p-2$, with leading coefficient
\begin{equation}\label{eq:tail-leading}
 \frac{c_p}{(-p-2)!}.
\end{equation}
For every $k>\max\{q+1,0\}$ and $n\ge k$, the lag-$k$ recurrence coefficient is
\begin{equation}\label{eq:tail-recurrence-coefficient}
 [p_{n-k}](xp_n)=-(-1)^kP_g(k)\fall{n}{k}.
\end{equation}
Consequently $P_g(k)$, and hence this coefficient, is nonzero for all sufficiently large $k$.  Every sufficiently deep lag therefore occurs, so the recurrence has no uniform finite lower bandwidth.  Conversely, the eventual signed coefficients determine $P_g$, and hence
\[
 p=-\deg P_g-2,
 \qquad
 c_p=(-p-2)!\,[t^{-p-2}]P_g(t).
\]
Thus the lowest Laurent term of $g$ is recoverable from the degree recurrence.
\end{theorem}

\begin{proof}
From \eqref{eq:h-def}
\[
 h(D)=\sum_{j=p}^{q}c_j(1+D)^{j+1}.
\]
For $j\ge-1$, the summand is a polynomial of degree $j+1$ and therefore contributes nothing to $h_k$ once $k>q+1$.  For $j\le-2$, the negative-binomial expansion gives
\[
 [D^k](1+D)^{j+1}
 =\binom{j+1}{k}
 =(-1)^k\binom{k-j-2}{-j-2}.
\]
Here and in \eqref{eq:tail-coefficients}, the binomial polynomials are interpreted over the prime subfield $\mathbb Q\subset\K$.  This proves \eqref{eq:tail-coefficients}.  Each summand of $P_g(t)$ has degree $-j-2$ and leading coefficient $c_j/(-j-2)!$.  The unique summand of largest degree comes from the lowest Laurent exponent $j=p$, proving \eqref{eq:tail-leading}.  Since $\K$ has characteristic zero, a nonzero polynomial has only finitely many roots among the embedded nonnegative integers.  The recurrence coefficient formula now follows from \eqref{eq:degree-recurrence}.
\end{proof}

\begin{corollary}\label{thm:filtration}
For every $m\ge0$, the following conditions are equivalent.
\begin{enumerate}[label=\textup{(\roman*)},leftmargin=2.3em]
\item The fixed-lowering triple belongs to the stencil layer $\calL_m$.
\item The Laurent shear satisfies
\[
 g(S)\in\Span_{\K}\{S^{-1},1,S,\ldots,S^{m-1}\}.
\]
\item The associated polynomial eigenbasis has uniform lower bandwidth at most $m$.
\end{enumerate}
For every $m\ge1$,
\begin{equation}\label{eq:exact-strata}
 (E,H,F)\in\calL_m\setminus\calL_{m-1}
 \quad\Longleftrightarrow\quad
 \text{$\{p_n\}$ has exact lower bandwidth $m$}.
\end{equation}
The minimal layer $\calL_0$ consists of translated falling-factorial systems and has bandwidth zero.
\end{corollary}

\begin{proof}
The equivalence of (i) and (ii) is Corollary~\ref{cor:all-layers}.  If (i) holds, Theorem~\ref{thm:numerical-bandwidth} gives $r=\nu_+(H)\le m$, so (iii) follows.  Conversely, if the bandwidth is at most $m$, Theorem~\ref{thm:numerical-bandwidth} gives $\nu_-(H)=-1$, $\nu_-(E)=-2$, $\nu_+(H)=r\le m$, and $\nu_+(E)=2r-1\le2m-1$; hence the triple belongs to $\calL_m$.  The exact-stratum assertion follows by replacing both weak inequalities with equality.  The description of $\calL_0$ is Corollary~\ref{cor:minimal-layer}.
\end{proof}

We now connect the recurrence terminology to the established $d$-orthogonal literature.  Let $d\ge1$ and let $\mathbf u=(u_0,\ldots,u_{d-1})^{\mathsf T}$ be a vector of linear functionals on $\K[x]$.  A monic sequence $\{p_n\}_{n\ge0}$ is called $d$-orthogonal with respect to $\mathbf u$ if, for $0\le j\le d-1$,
\begin{equation}\label{eq:d-orthogonality-definition}
 \langle u_j,p_m p_n\rangle=0\quad\text{whenever }m>nd+j,
 \qquad
 \langle u_j,p_n p_{nd+j}\rangle\ne0
 \quad(n\ge0).
\end{equation}
The nonvanishing conditions are the regularity conditions.  In the indexing used here, the generalized Favard theorem states that a monic sequence is a regular $d$-orthogonal polynomial sequence if and only if multiplication by $x$ has a recurrence
\begin{equation}\label{eq:d-favard}
 xp_n=p_{n+1}+\sum_{k=0}^{d}\beta_{n,k}p_{n-k},
 \qquad
 \beta_{n,d}\ne0\quad(n\ge d),
\end{equation}
with $p_j=0$ for $j<0$.  This is the monic recurrence and regularity condition of \cite[Definition~1.1 and (1.2)]{BenCheikhZaghouani2003}; the generalized Favard theorem in this normalization is due to Maroni \cite{Maroni1989} (see also \cite{Maroni1981,VanIseghem1987}).

\begin{corollary}\label{cor:d-orthogonal}
For fixed-lowering triples, the following nested loci are exact
\[
 \K[S,S^{-1}]
 \supset
 S^{-1}\K[S]
 \supset
 \Span_{\K}\{S^{-1},1\}
 \supset
 \K S^{-1}.
\]
They correspond respectively to
\begin{enumerate}[label=\textup{(\roman*)},leftmargin=2.3em]
\item all finite Laurent-shift fixed-lowering triples;
\item triples whose eigenbasis has finite lower bandwidth;
\item triples whose eigenbasis has a recurrence with at most three terms;
\item translated falling-factorial systems with a two-term recurrence.
\end{enumerate}
If $g(S)=\sum_{j=-1}^{q}c_jS^j$ with $q\ge0$ and $c_q\ne0$, then the basis is a regular $(q+1)$-orthogonal polynomial sequence.  Equivalently, the exact one-sided layer $\calL_d\setminus\calL_{d-1}$ is the regular $d$-orthogonal $\Delta$-Appell locus for every $d\ge1$.
\end{corollary}

\begin{proof}
The nested loci follow from Theorems~\ref{thm:classification}, \ref{thm:recurrence}, and \ref{thm:filtration}.  If $q\ge0$, equation \eqref{eq:degree-recurrence} has lower bandwidth $d=q+1$.  In the notation of \eqref{eq:d-favard}, its deepest coefficient is $\beta_{n,d}=-c_q\fall{n}{d}$, which is nonzero for every $n\ge d$.  All coefficients outside the displayed band vanish by \eqref{eq:degree-recurrence}, and the deepest coefficient is nonzero for every $n\ge d$.  These are precisely the band and regularity hypotheses in \eqref{eq:d-favard}; the generalized Favard theorem therefore gives regular $d$-orthogonality.
\end{proof}

\begin{table}[t]
\centering
\caption{Operator and polynomial loci determined by the Laurent shear.}
\label{tab:loci}
\footnotesize
\renewcommand{\arraystretch}{1.2}
\begin{tabularx}{\textwidth}{>{\raggedright\arraybackslash}p{.27\textwidth}>{\raggedright\arraybackslash}p{.27\textwidth}X}
\toprule
Condition on $g(S)$ & Lattice-side consequence & Degree-side consequence\\
\midrule
$g\in\K[S,S^{-1}]$
& A finite-order Laurent-shift triple, possibly with two-sided support
& The general fixed-lowering locus; the recurrence is infinite exactly outside $S^{-1}\K[S]$.\\
$g\in S^{-1}\K[S]$
& The triple lies in some one-sided layer $\calL_m$
& Finite lower bandwidth; these are the finite-band $\Delta$-Appell systems.\\
$g\in\Span\{S^{-1},1,\ldots,S^{m-1}\}$
& The triple lies in $\calL_m$
& Bandwidth at most $m$.\\
$[S^{m-1}]g\ne0$ in the preceding row
& Exact one-sided layer $\calL_m\setminus\calL_{m-1}$
& Exact bandwidth $m$; regular $m$-orthogonality for $m\ge1$.\\
$g=cS^{-1}$
& Minimal layer $\calL_0$
& Translated falling factorials with a two-term recurrence.\\
$g=c_{-1}S^{-1}+c_0$, $c_0<0$
& Positive-orthogonality locus in $\calL_1$; $H$ has three-point support
& Translated monic Charlier systems with a positive three-term recurrence.\\
\bottomrule
\end{tabularx}
\end{table}

In the finite-band case let $r$ be the exact lower bandwidth and define a finite-order difference operator in the degree variable by
\begin{equation}\label{eq:Bg}
 (B_gu)_n=u_{n+1}+(n-h_0)u_n-
 \sum_{k=1}^{r}h_k\fall{n}{k}u_{n-k},
\end{equation}
with $u_j=0$ for $j<0$.  Then
\[
 B_g\bigl(p_\bullet(x)\bigr)=x p_\bullet(x),
 \qquad
 Hp_n(x)=(2n-\lambda)p_n(x).
\]
These identities give simultaneous finite-order spectral equations in the lattice and degree variables.  We use them only in this elementary sense, that is to say no classification of commuting operator algebras or stronger bispectral structure is claimed.

\begin{remark}\label{rem:asymmetry}
The exact lattice endpoints in Theorem~\ref{thm:support} and the recurrence bandwidth in Theorem~\ref{thm:recurrence} are controlled by different ends of the Laurent support of $g$.  Once no powers below $S^{-1}$ occur, the rightmost exponent determines the recurrence length.  Any power below $S^{-1}$ destroys uniform degree bandwidth, although $H$ and $E$ remain finite Laurent-shift operators.  Thus the complement of $\bigcup_m\calL_m$ inside the full Laurent parameter space is exactly the finite-lattice-support/infinite-recurrence region.
\end{remark}

\section{Positive orthogonality and the Charlier locus}

In this section $\K=\R$, so the triple and its monic eigenbasis have real coefficients.  After projecting away the independent parameter $\lambda$, positive orthogonality selects an open half-plane in the two-dimensional shear space $\calL_1$.  The additional condition is a sign restriction on the constant Laurent coefficient.

\begin{theorem}\label{thm:positive}
Let $\{p_n\}_{n\ge0}$ be the monic eigenbasis of a fixed-lowering triple.  The following are equivalent.
\begin{enumerate}[label=\textup{(\roman*)},leftmargin=2.3em]
\item The sequence $\{p_n\}$ is orthogonal with respect to a positive Borel measure on $\R$ having infinite support and finite moments of all orders.
\item There exist $a>0$ and $\beta\in\R$ such that
\begin{equation}\label{eq:charlier-g}
 g(S)=-a+(a+\beta)S^{-1}.
\end{equation}
\item The shear belongs to $\calL_1$,
\[
 g(S)=c_{-1}S^{-1}+c_0,
\]
and its constant coefficient satisfies $c_0<0$.
\end{enumerate}
In this case
\begin{equation}\label{eq:charlier-egf}
 \sum_{n=0}^{\infty}p_n(x)\frac{z^n}{n!}
 =e^{-az}(1+z)^{x+a+\beta},
\end{equation}
\begin{equation}\label{eq:charlier-recurrence}
 xp_n=p_{n+1}+(n-\beta)p_n+anp_{n-1},
\end{equation}
and
\begin{equation}\label{eq:charlier-identification}
 p_n(x)=C_n(x+a+\beta;a),
\end{equation}
where $C_n(\,\cdot\,;a)$ denotes the monic Charlier polynomial.  Up to multiplication by a positive constant, the orthogonality measure is unique.  It is supported at
\[
 x=k-a-\beta,
 \qquad k=0,1,2,\ldots,
\]
with weights proportional to $e^{-a}a^k/k!$.
\end{theorem}

\begin{proof}
Assume (i).  By Favard's theorem \cite[Chapter~I, Section~4]{Chihara1978}, multiplication by $x$ has a three-term recurrence
\[
 xp_n=p_{n+1}+b_np_n+\gamma_np_{n-1},
 \qquad \gamma_n>0.
\]
Expansion in the basis $\{p_j\}$ is unique.  Comparing with Theorem~\ref{thm:recurrence} therefore gives $h_k=0$ for every $k\ge2$, so $h(D)=h_0+h_1D$.  The coefficient of $p_{n-1}$ is $-h_1n$, and positivity for every $n\ge1$ forces $h_1<0$.  Write $h_0=\beta$ and $h_1=-a$ with $a>0$.  Then
\[
 g(D)=\frac{h(D)}{1+D}
 =\frac{\beta-aD}{1+D}
 =-a+\frac{a+\beta}{1+D},
\]
which is \eqref{eq:charlier-g}.

Conversely, assume (ii).  Then $h(D)=Sg(S)=\beta-aD$, and Theorem~\ref{thm:recurrence} gives \eqref{eq:charlier-recurrence}.  Moreover,
\[
 G'(z)=-a+\frac{a+\beta}{1+z},
\]
so $G(0)=0$ gives
\[
 G(z)=-az+(a+\beta)\log(1+z).
\]
Equation \eqref{eq:charlier-egf} follows from Theorem~\ref{thm:eigenbasis}.  This is the standard exponential generating function for the monic Charlier polynomials \cite{KoekoekLeskySwarttouw2010,Meixner1934}, proving \eqref{eq:charlier-identification}.  The monic Charlier polynomials $C_n(y;a)$ are orthogonal for the Poisson measure supported at $y=k$ with weights $e^{-a}a^k/k!$.  Translating $y=x+a+\beta$ gives an orthogonality measure with the stated support and weights.

It remains to justify uniqueness of the positive measure.  In the monic recurrence \eqref{eq:charlier-recurrence}, the subdiagonal coefficient is $\gamma_n=an$.  Hence the off-diagonal Jacobi coefficient for the corresponding orthonormal system is $\sqrt{\gamma_n}=\sqrt{an}$.  Since
\[
 \sum_{n=1}^{\infty}\frac{1}{\sqrt{an}}=\infty,
\]
Carleman's criterion for the Hamburger moment problem implies determinacy; see, for example, \cite{Chihara1978}.  Therefore, every positive orthogonality measure in (i) is a positive scalar multiple of the translated Poisson measure above.

Finally, (ii) and (iii) are equivalent under
\[
 c_0=-a,
 \qquad c_{-1}=a+\beta,
 \qquad
 a=-c_0,
 \qquad \beta=c_{-1}+c_0.
\]
\end{proof}

\begin{corollary}\label{cor:charlier-equation}
Under the equivalent conditions of Theorem~\ref{thm:positive},
\[
 \frac{H+\lambda}{2}
 =-aS+(x+2a+\beta)-(x+a+\beta)S^{-1}.
\]
Hence $Hp_n=(2n-\lambda)p_n$ is equivalent to
\begin{equation}\label{eq:translated-charlier-equation}
 -ap_n(x+1)+(x+2a+\beta)p_n(x)
 -(x+a+\beta)p_n(x-1)=np_n(x).
\end{equation}
After the translation $x'=x+a+\beta$, this becomes
\[
 -au(x'+1)+(x'+a)u(x')-x'u(x'-1)=nu(x'),
\]
the classical Charlier difference equation.
\end{corollary}

\begin{proof}
Substitute \eqref{eq:charlier-g} into \eqref{eq:H-shift-form} and collect the coefficients of $S$, $1$, and $S^{-1}$.
\end{proof}

\begin{remark}
Meixner's orthogonal-Sheffer classification is much broader than Theorem~\ref{thm:positive}.  The point here is the exact intersection is created by fixing the delta operator $D=S-1$ and requiring the full $\mathfrak{sl}_2$-triple to lie in the finite Laurent-shift algebra.  Within this normalized class, and after suppressing the independent $\lambda$-coordinate, positive orthogonality is the open half-plane $c_0<0$ in the $(c_{-1},c_0)$-parameter plane.  The full locus of triples is its product with the $\lambda$-line.
\end{remark}

\section{Examples}

\begin{example}
Let $g(S)=cS^{-1}$.  Then
\[
 Y=(x+c)S^{-1},
 \qquad
 p_n(x)=(x+c)^{\underline n},
\]
and
\[
 xp_n=p_{n+1}+(n-c)p_n.
\]
This is the minimal one-sided layer $\calL_0$.  Its recurrence is two-term and therefore does not define a nondegenerate positive orthogonality measure of infinite support.
\end{example}

\begin{example}
Let
\[
 g(S)=-a+(a+\beta)S^{-1},
 \qquad a>0.
\]
Then $H$ is a three-point difference operator, the degree recurrence is three-term, and the eigenbasis is the translated monic Charlier system of Theorem~\ref{thm:positive}.
\end{example}

\begin{example}
Let $m\ge1$ and $g(S)=cS^m$ with $c\ne0$.  Then
\[
 H=2x-2xS^{-1}+2c(S^{m+1}-S^m)-\lambda,
\]
and
\[
 \sum_{n\ge0}p_n(x)\frac{z^n}{n!}
 =\exp\!\left(\frac{c}{m+1}\bigl((1+z)^{m+1}-1\bigr)\right)(1+z)^x.
\]
Thus the rightmost exponent $m$ gives the exact stratum $\calL_{m+1}\setminus\calL_m$ and lower bandwidth $m+1$; the recurrence runs from $p_{n+1}$ to $p_{n-m-1}$.
\end{example}

\begin{example}\label{ex:quadratic-band}
Let $g(S)=cS$ with $c\ne0$.  Then the exact lower bandwidth is two and
\[
 H=2cS^2-2cS+(2x-\lambda)-2xS^{-1}.
\]
Using \eqref{eq:E-expand} gives the complete raising operator
\[
\begin{aligned}
 E={}&c^2S^3-c^2S^2
 +\bigl(2cx+c-\lambda c\bigr)S
 -(2cx+c)\\
 &+x(x-1-\lambda)S^{-1}-x(x-1)S^{-2}.
\end{aligned}
\]
Thus this is an explicit finite-stencil Cartan/raising system rather than only an abstract shear.  The eigenpolynomials have generating function
\[
 \sum_{n\ge0}p_n(x)\frac{z^n}{n!}
 =\exp\!\left(cz+\frac{c}{2}z^2\right)(1+z)^x
\]
and satisfy
\[
 xp_n=p_{n+1}+(n-c)p_n-2cnp_{n-1}
 -c\fall{n}{2}p_{n-2}.
\]
The associated scalar lattice equation is
\[
 x\bigl(p_n(x)-p_n(x-1)\bigr)
 +c\bigl(p_n(x+2)-p_n(x+1)\bigr)=np_n(x).
\]
This example lies in $\calL_2\setminus\calL_1$ and is a regular $2$-orthogonal $\Delta$-Appell system.
\end{example}

\begin{example}\label{ex:negative-power}
Let $r\ge2$ and $g(S)=cS^{-r}$ with $c\ne0$.  Then $H$ and $E$ remain finite-order Laurent-shift operators, but
\[
 h(D)=Sg(S)=c(1+D)^{1-r}
 =c\sum_{k=0}^{\infty}(-1)^k\binom{r+k-2}{k}D^k.
\]
Thus Theorem~\ref{thm:recurrence} gives the explicit recurrence
\begin{equation}\label{eq:infinite-tail-example}
 xp_n=p_{n+1}+(n-c)p_n
 -c\sum_{k=1}^{n}(-1)^k\binom{r+k-2}{k}
 \fall{n}{k}p_{n-k}.
\end{equation}
Every depth $k$ occurs once $n\ge k$, so there is no uniform finite lower bandwidth.  On the lattice side, however, Theorem~\ref{thm:support} gives
\[
 \nu_-(H)=-r,
 \quad \nu_+(H)=0,
 \qquad
 \nu_-(E)=-2r,
 \quad \nu_+(E)=-1.
\]
The eigenpolynomials themselves are unremarkable: integrating $G'(z)=c(1+z)^{-r}$ with $G(0)=0$ gives
\[
 \sum_{n\ge0}p_n(x)\frac{z^n}{n!}
 =\exp\Bigl(\frac{c}{1-r}\bigl((1+z)^{1-r}-1\bigr)\Bigr)(1+z)^x,
\]
a closed-form $\Delta$-Appell generating function.  This is an explicit separation between finite Laurent-shift support and finite degree-side recurrence: the former holds and the latter fails.
\end{example}

\begin{example}\label{ex:mixed-shear}
Let
\[
 g(S)=\alpha S^{-2}+\beta+\gamma S^2,
 \qquad
 \alpha\gamma\ne0.
\]
The lowest exponent $-2$ and highest exponent $2$ control opposite ends of the lattice stencil:
\[
 \nu_-(H)=-2,
 \quad \nu_+(H)=3,
 \qquad
 \nu_-(E)=-4,
 \quad \nu_+(E)=5.
\]
On the degree side,
\[
 h(D)=Sg(S)=\alpha(1+D)^{-1}+\beta(1+D)+\gamma(1+D)^3.
\]
Thus $h_k=\alpha(-1)^k$ for every $k\ge4$, and the recurrence contains the nonzero term
\[
 -\alpha(-1)^k\fall{n}{k}p_{n-k}
\]
for every $k\ge4$ and $n\ge k$.  The coefficient $\alpha$ determines the infinite left recurrence tail, while $\gamma$ determines the right stencil endpoints of $H$ and $E$.
\end{example}

\section{Conclusion and outlook}

The fixed-lowering inverse problem has a complete answer in the Laurent-shift algebra.  Prescribing $F=1-S$ reduces every completion to a unique Laurent shear, but the significance of that parameter is support-theoretic in that it simultaneously records the exact lattice endpoints, the polynomial eigenbasis, and the locality or nonlocality of multiplication by $x$ in degree. The one-sided and complementary regions behave sharply differently.  For $g\in S^{-1}\K[S]$, finite recurrence bandwidth is equivalent to the minimal possible left endpoints of both $H$ and $E$, and the right endpoint of $H$ is exactly the bandwidth.  Below $S^{-1}$, finite lattice locality survives while degree locality fails.  The eventual recurrence tail nevertheless retains exact information since its polynomial profile recovers the lowest Laurent term.  Thus, the passage from lattice stencils to degree recurrences loses finiteness but not all support data.

The unit-step normalization is inessential and may be rescaled to a nonzero forward step.  A genuinely different problem begins when the prescribed lowering operator already has two-sided support.  For example, consider
\[
 F=-\frac{S-S^{-1}}{2}.
\]
Heisenberg pairs associated with central and more general finite-difference schemes have been studied in \cite{GorskiSzmigielski1998}.  The additional problem suggested by the present paper is more rigid since we classify the completions of this actual two-sided lowering operator inside the finite Laurent-shift algebra, while retaining its stencil data, and determine whether degree-side recurrence locality can still be read from intrinsic endpoint information.  We leave that fixed-lowering classification for future work.

\section*{Acknowledgements}

The author would like to thank Paul R.~Garvey for insightful discussions and input on the finite-difference calculus during this work. 

\section*{Disclosure Statement}

The author reports no conflict of interest.

\section*{Declaration of generative AI use}

The author reports generative AI was not used in their research or preparation of this manuscript.

\end{document}